\documentclass{amsart}

\usepackage{tikz}
\usepackage{xcolor}

\usepackage{latexsym,amsmath,amssymb,epic}
\usepackage{latexsym,amssymb,epic}
\usepackage{amsmath,amsthm}
\usepackage{color}

\usepackage{soul}

\usepackage{hyperref}
\hypersetup{
	colorlinks=true, %set true if you want colored links
	linktoc=all, %set to all if you want both sections and subsections linked
	linkcolor=blue} %choose some color if you want links to stand out

\newtheorem{theorem}{Theorem}[section]
\newtheorem{lemma}[theorem]{Lemma}

\begin{document}

	\title[A Partition Function of Dombos]{Arithmetic Properties Satisfied by a Recent Integer Partition Function of Dombos}

	\author{Robson da Silva}
	\address{Universidade Federal de S\~ao Paulo, S\~ao Jos\'e dos Campos, SP 12247--014, Brazil}
	\email{silva.robson@unifesp.br}

	\author{James A. Sellers}
	\address{Department of Mathematics and Statistics, University of Minnesota Duluth, Duluth, MN 55812, USA}
	\email{jsellers@d.umn.edu}

	
	
	%\footnotetext[1]{} \vspace{0.5in}
	%\footnotetext[1]{2010 AMS Classification Numbers: Primary, 11P81; Secondary, 05A17.}

\begin{abstract}
In recent work of Dombos, the set of integer partitions of $n$ wherein the parts are either divisible by 4 or congruent to $\pm 1 \pmod{6}$ arose in a natural way.   In this work, we will denote the function which counts the number of such partitions of $n$ by $dp(n)$.  Using elementary generating function manipulations and classical $q$--series results, we prove several congruences satisfied by $dp(n)$.  As an example, we prove that, for all $\alpha \geq 1$ and all $n \geq  0$,
\begin{equation*}
    dp \left( 3^{2\alpha + 1}n + \frac{7 \cdot 9^\alpha + 1}{4} \right) \equiv 0 \pmod{3}.
\end{equation*}

\end{abstract}

	\maketitle
	
	%\noindent{\footnotesize{\bf Abstract.}

\section{Introduction} 

A partition of a positive integer $n$ is a nonincreasing sequence of positive integers $\lambda_1, \lambda_2, \dots, \lambda_r$ such that $n=\lambda_1+ \lambda_2+ \dots+ \lambda_r$.  We say that each $\lambda_i$, $1\leq i\leq r$, is called a part of the partition.  We denote the number of partitions of $n$ by $p(n)$.  

In the late 1910's, Ramanujan \cite{Ramanujan} proved the following congruences satisfied by $p(n)$ for the moduli $5, 7,$ and $11$ using generating function manipulations. 

\begin{theorem}%\label{ramanujanCongruences}
    For all $n\geq 0$, 
    \begin{align*}
        p(5n+4) &\equiv 0 \pmod{5}, \\
        p(7n+5) &\equiv 0 \pmod{7}, \text { and}\\
        p(11n+6) &\equiv 0 \pmod{11}.
    \end{align*}
\end{theorem}
These congruences, and others that have since been proven for numerous restricted partition functions, serve as the motivation for this work.  

In recent work of Dombos \cite{Dombos}, the set of integer partitions of $n$ wherein the parts are either divisible by 4 or congruent to $\pm 1 \pmod{6}$ arose in a natural way.  (See \cite[Theorem 1.6]{Dombos} for more details.)  In this work, we will denote the function which counts the number of such partitions of $n$ by $dp(n)$.  In light of this definition, it is clear that the generating function for $dp(n)$ is given by the following:  
\begin{equation}
G(q):=\sum_{n=0}^{\infty} dp(n)q^n = \frac{1}{(q;q^6)_\infty (q^5;q^6)_\infty (q^4;q^4)_\infty},
\label{gf1}
\end{equation}
where $(a;q)_\infty := (1-a)(1-aq)(1-aq^{2})\cdots$
is the usual $q$-Pochhammer symbol. 

The primary goal of this brief note is to prove a number of Ramanujan--like congruences satisfied by the function $dp(n)$ in certain arithmetic progressions.  In the work below, we will prove the following results. 

\begin{theorem}
For all $n \geq 0$,
	\begin{align}
	dp(6n+4) & \equiv 0 \pmod{2}, \label{c1} \\ 
%    dp(27n+16) & \equiv 0 \pmod{3}, \label{c2} \\
    dp(18n+10) & \equiv 0 \pmod{4}, \label{c3} \\
    dp(54n+52) &  \equiv 0 \pmod{8}. \label{c4}
    \end{align}
\label{Th1}
\end{theorem}

\begin{theorem} Let $p$ be a prime such that $p \equiv 17 \mbox{ or } 23 \pmod{24}$. Then for all $k,m \geq 0$ with $p \nmid m$, we have
\begin{equation*}
dp\left( 6p^{2k+1}m + \frac{15p^{2k+2}+1}{4} \right) \equiv 0 \pmod{4}.
\end{equation*}
\label{Th2}
\end{theorem}

\begin{theorem}
For all $n \geq 0$,
\begin{align}
dp(27n+16) & \equiv 0 \pmod{3}.
\label{c2}
\end{align}
\label{Th3}
\end{theorem}

\begin{theorem}
For all $n \geq 0$,
\begin{equation*}
dp(27n+7) \equiv dp(3n+1) \pmod{3}.
\end{equation*}
\label{internal_cong}
\end{theorem}

\begin{theorem}
For all $\alpha \geq 1$ and all $n \geq  0$,
\begin{equation*}
    dp \left( 3^{2\alpha + 1}n + \frac{7 \cdot 9^\alpha + 1}{4} \right) \equiv 0 \pmod{3}.
\end{equation*}
\label{Th4}
\end{theorem}

\section{Preliminaries}

Throughout the remainder of this paper, we define 
		$$f_k := (q^k;q^k)_{\infty}$$
in order to shorten the notation.

We begin by rewriting the generating function \eqref{gf1} in the following form.

\begin{lemma}
\begin{equation}
\sum_{n=0}^{\infty} dp(n)q^n = \frac{f_2f_3}{f_1f_4f_6}.
\label{gf2}
\end{equation}
\end{lemma}

\begin{proof}
Initially, we note that
\begin{align*}
\frac{1}{(q;q^6)_\infty (q^5;q^6)_\infty} & = \frac{(q^2;q^6)_\infty(q^3;q^6)_\infty(q^4;q^6)_\infty(q^6;q^6)_\infty}{(q;q^6)_\infty(q^2;q^6)_\infty(q^3;q^6)_\infty(q^4;q^6)_\infty(q^5;q^6)_\infty(q^6;q^6)_\infty} \\
& = \frac{(q^3;q^3)_\infty(q^2;q^6)_\infty(q^4;q^6)_\infty}{(q;q^3)_\infty(q^2;q^3)_\infty(q^3;q^3)_\infty} \\
& = \frac{(q^2;q^6)_\infty(q^4;q^6)_\infty}{(q;q^3)_\infty(q^2;q^3)_\infty} \\
& = \frac{(q^2;q^2)_\infty}{(q^6;q^6)_\infty}\cdot \frac{(q^3;q^3)_\infty}{(q;q)_\infty} \\
& = \frac{f_2f_3}{f_1f_6}.  
%& = \frac{(-q;q^3)_\infty(-q^2;q^3)_\infty(q;q^3)_\infty(q^2;q^3)_\infty}{(q;q^3)_\infty(q^2;q^3)_\infty} \\
%& = (-q;q^3)_\infty(-q^2;q^3)_\infty.
\end{align*}
Thus, it follows from \eqref{gf1} that
\begin{align*}
 \sum_{n=0}^{\infty} dp(n)q^n 
 %& = \frac{1}{f_4} (-q;q^3)_\infty(-q^2;q^3)_\infty \\
 %& = \frac{1}{f_4} \frac{(-q;q^3)_\infty(-q^2;q^3)_\infty (-q^3;q^3)_\infty}{(-q^3;q^3)_\infty} \\
 %& = \frac{1}{f_4} \frac{(-q;q)_\infty}{(-q^3;q^3)_\infty} \\
 & = \frac{f_2f_3}{f_1f_4f_6},
\end{align*}
which is \eqref{gf2}.
\end{proof}

The proofs of the main results mentioned above require a few well-known 2- and 3-dissections.

\begin{lemma}[\cite{R}, Lemma 1]
We have
\begin{align}
{f_{1}^{2}} & =  \frac{f_{2}f_8^5}{f_{4}^{2}f_{16}^{2}} -2q\frac{f_{2}f_{16}^{2}}{f_{8}}, \label{eq1061} \\
\frac{1}{f_{1}^{2}} & =  \frac{f_{8}^{5}}{f_{2}^{5}f_{16}^{2}} + 2q\frac{f_{4}^{2}f_{16}^{2}}{f_{2}^{5}f_{8}}.   
\label{eq2}
\end{align}
\end{lemma}

%\begin{lemma}[\cite{H}, Eq. (30.10.3)]
%We have
%\begin{equation}
%\frac{f_3}{f_1} = \frac{f_4f_6f_{16}f_{24}^2}{f_2^2f_8f_{12}f_{48}} + q \frac{f_6f_8^2f_{48}}{f_2^2f_{16}f_{24}}.   
%\label{eq3}
%\end{equation}
%\end{lemma}

%\begin{lemma}[\cite{Yao}, Eq. (2.19)]
%We have
%\begin{equation}
%\frac{f_{4}}{f_{1}}  = \displaystyle\frac{f_{12}f_{18}^{4}}{f_{3}^3f_{36}^2} + q\displaystyle\frac{f_{6}^2f_{9}^{3}f_{36}}{f_{3}^4f_{18}^2} + 2q^2 \displaystyle\frac{f_{6}f_{18}f_{36}}{f_{3}^3}.   
%\label{eq4}
%\end{equation}
%\end{lemma}

\begin{lemma}[\cite{Toh}, Eq. (2.1c)]
We have
\begin{equation}
\frac{f_{2}}{f_{1}f_{4}}  = \displaystyle\frac{f_{18}^{9}}{f_{3}^2f_{9}^3f_{12}^2f_{36}^3} + q\displaystyle\frac{f_{6}^2f_{18}^{3}}{f_{3}^3f_{12}^3} +q^2 \displaystyle\frac{f_{6}^4f_{9}^3f_{36}^{3}}{f_{3}^4f_{12}^4f_{18}^3}.   
\label{eq5}
\end{equation}
\end{lemma}

\begin{lemma}[\cite{Toh}, Eq. (2.1b)]
We have
\begin{equation}
\frac{f_{2}}{f_{1}^2} = \frac{f_{6}^4f_{9}^6}{f_{3}^8f_{18}^3} + 2q\displaystyle\frac{f_{6}^3f_{9}^3}{f_{3}^7} + 4q^2\displaystyle\frac{f_{6}^2f_{18}^3}{f_{3}^6}.   
\label{eq6}
\end{equation}
\end{lemma}

\begin{lemma}[\cite{B1}, Eq. (22.4)]
We have
\begin{equation}
\frac{f_{1}^2}{f_{2}} = \displaystyle\frac{f_{9}^2}{f_{18}} - 2q\displaystyle\frac{f_{3}f_{18}^2}{f_{6}f_{9}}.   
\label{eq7}
\end{equation}
\end{lemma}

%\begin{lemma}[\cite{H}, Eq. (14.8.5)]
%We have
%\begin{equation}
%f_{1}^{3}  =  \displaystyle\frac{f_{6}f_{9}^6}{f_{3}f_{18}^3} + 4q^3\displaystyle\frac{f_{3}^2f_{18}^{6}}{f_{6}^2f_{9}^3} - 3qf_{9}^{3}.   
%\label{eq8}
%\end{equation}
%\end{lemma}

\begin{lemma}[\cite{Toh}, Eq. (3.1)] We have
\begin{equation}
\displaystyle\frac{f_{2}^3}{f_{1}^3} = \displaystyle\frac{f_{6}}{f_{3}} +3q\displaystyle\frac{f_{6}^4f_{9}^{5}}{f_{3}^8f_{18}} +6q^2\displaystyle\frac{f_{6}^3f_{9}^{2}f_{18}^{2}}{f_{3}^7} +  12q^3\displaystyle\frac{f_{6}^2f_{18}^5}{f_{3}^6f_{9}}.  
\label{eq1062}
\end{equation}
\end{lemma}

We recall Ramanujan's theta functions
\begin{align*}
	f(a,b) & :=\sum_{n=-\infty}^\infty a^\frac{n(n+1)}{2}b^\frac{n(n-1)}{2}, \mbox{ for } |ab|<1, \nonumber \\
	\phi(q) & := f(q,q) = \sum_{n=-\infty}^{\infty} q^{n^2} = \frac{f_2^5}{f_1^2f_4^2}, \\
	\psi(q) & := f(q,q^3) = \sum_{n=0}^{\infty} q^{n(n+1)/2} = \frac{f_2^2}{f_1}. 
\end{align*}
With the above definitions, we can state the following 3-dissection of $1/\psi(q)$.

%\begin{lemma}[\cite{Shane}, Eq. (14)]
%We have
%\begin{align}
%\frac{1}{f_{1}^{3}} & = \displaystyle\frac{f_{9}^{3}}{f_{3}^{12}} \left( f_{3}^2 \frac{\phi(-q^9)^6}{\phi(-q^3)^2} + 8q^3f_{3}^2 \frac{\phi(-q^9)^3\psi(q^9)^3}{\phi(-q^3)\psi(q^3)} + 16q^6f_{3}^2\frac{\psi(q^9)^6}{\psi(q^3)^2} \right. \label{eq9} \\
%& \ \ \ \ \ \ \ \ \ \ \left. + 3qf_{3}f_9^3\frac{\phi(-q^9)^3}{\phi(-q^3)} + 12q^4f_{3}f_9^3 \frac{\psi(q^9)^3}{\psi(q^3)} +9q^2f_9^6 \right).  \nonumber 
%\end{align}
%\end{lemma}

\begin{lemma}[\cite{Shane}, Lemma 2.2] We have
	\begin{equation}
	\frac{1}{\psi(q)} = \frac{f_3^2f_9^3}{f_6^6} - q\frac{f_3^3f_{18}^3}{f_6^7} + q^2\frac{f_3^4f_{18}^6}{f_6^8f_9^3}.
	\label{3diss1/psi}
	\end{equation}
\end{lemma}

Lastly, we require the classical identity of Jacobi (see \cite[Theorem 1.3.9]{B1})
\begin{equation}
	f_1^3 = \sum_{n=0}^{\infty} (-1)^n (2n+1) q^{n(n+1)/2}.
\label{Jacobi}
\end{equation}
%and Euler's identity (see \cite[Eq. (1.6.1)]{H})
%\begin{equation}
%	f_1 = \sum_{n=-\infty}^{\infty} (-1)^n q^{n(3n-1)/2}.
%	\label{Euler}
%\end{equation}
%%%%%%%%%%%%%%%%%%%%%%%%%%%%%%%%%%%%
%%%%%%%%%%%%%%%%%%%%%%%%%%%%%%%%%%%%
%%%%%%%%%%%%%%%%%%%%%%%%%%%%%%%%%%%%

\section{Proof of Theorem \ref{Th1}}

Let 
$$H(q) = \sum_{n=0}^{\infty}h(n)q^n = G(-q) = \sum_{n=0}^{\infty}dp(n)(-q)^n.$$
Thus, $h(n) = (-1)^ndp(n)$ for all $n$. We note that
$$(-q;-q)_\infty = \frac{f_2^3}{f_1f_4}.$$
Using this identity and \eqref{gf2} we find that
\begin{equation*}
H(q) = \frac{f_1f_6^2}{f_2^2f_3f_{12}} = \frac{f_6^2}{f_3f_{12}} \frac{1}{\psi(q)}.
\end{equation*}
Thanks to \eqref{3diss1/psi} we obtain
\begin{align*}
\sum_{n=0}^{\infty}h(3n+1)q^{3n+1} & = -q \frac{f_3^2f_{18}^3}{f_6^5f_{12}},
\end{align*}
which, after dividing by $q$ and replacing $q^3$ by $q$, yields
\begin{align*}
\sum_{n=0}^{\infty}h(3n+1)q^{n} & = - \frac{f_1^2f_{6}^3}{f_2^5f_{4}}.
\end{align*}   
Using \eqref{eq1061}, we can 2-dissect the expression above to obtain
\begin{align*}
\sum_{n=0}^{\infty}h(6n+4)q^{2n+1} & = 2q \frac{f_6^3f_{16}^2}{f_2^4f_{4}f_8}.
\end{align*}
Thus,
\begin{align*}
\sum_{n=0}^{\infty}h(6n+4)q^{n} & = 2 \frac{f_3^3f_{8}^2}{f_1^4f_{2}f_4},
\end{align*}
which proves \eqref{c1} since $dp(6n+4) = h(6n+4) \equiv 0 \pmod{2}$.

Working modulo 8, we have
\begin{align}
\sum_{n=0}^{\infty}h(6n+4)q^{n} & = 2 \frac{f_3^3f_{8}^2}{f_1^4f_{2}f_4} \equiv 2 f_3^3\frac{f_4^3}{f_2^3} \pmod{8}.
\label{eq1063}
\end{align}
Thanks to \eqref{eq1062} we have
\begin{align*}
\sum_{n=0}^{\infty}h(18n+10)q^{3n+1} & \equiv 4q^4 \frac{f_3^3f_{12}^3f_{18}^2f_{36}^2}{f_6^7} \pmod{8},
\end{align*}
which simplifies to
\begin{align*}
\sum_{n=0}^{\infty}h(18n+10)q^{n} & \equiv 4q \frac{f_1^3f_{4}^3f_{6}^2f_{12}^2}{f_2^7} \pmod{8}.
\end{align*}
It follows that $h(18n+10) \equiv 0 \pmod{4}$, which completes the proof of \eqref{c3} since $h(n) = (-1)^n dp(n)$.

In order to prove \eqref{c4}, we use \eqref{eq1062} to extract the terms of the form $q^{3n+2}$ from \eqref{eq1063}. The resulting congruence is
\begin{align}
\sum_{n=0}^{\infty}h(18n+16)q^{n} & \equiv 6\frac{f_1^3f_{4}^4f_{6}^5}{f_2^8f_{12}}  \pmod{8}\nonumber \\
& \equiv 6\frac{f_6^5}{f_{12}} f_1^3 \pmod{8} \nonumber \\
& \equiv 6\frac{f_6^5}{f_{12}} \sum_{n=0}^{\infty} (-1)^n (2n+1) q^{n(n+1)/2} \pmod{8} \label{eq1064}
\end{align}
by \eqref{Jacobi}. We note that $f_6^5/f_{12}$ is a function of $q^3$ and $n(n+1)/2$ is never congruent to 2 modulo 3. Hence, once we expand \eqref{eq1064} as a power series in $q$, there will be no powers of the form $q^{3n+2}$. Therefore, for all $n \geq 0$, 
$$h(18(3n+2)+16) = h(54n+52) \equiv 0 \pmod{8},$$
which completes the proof of \eqref{c4}.

%%%%%%%%%%%%%%%%%%%%%%%%%%%%%%%%%%%%%%%%%%%
%%%%%%%%%%%%%%%%%%%%%%%%%%%%%%%%%%%%%%%%%%%
%%%%%%%%%%%%%%%%%%%%%%%%%%%%%%%%%%%%%%%%%%%

\section{Proof of Theorem \ref{Th2}}

Initially, we use \eqref{eq5} to extract the terms of the form $q^{3n+1}$ from \eqref{gf2}, which gives
$$\sum_{n=0}^{\infty} dp(3n+1)q^{3n+1} = q\frac{f_{6}f_{18}^3}{f_3^2f_{12}^3}.$$
After dividing by $q$ and replacing $q^3$ by $q$, we are left with
\begin{align}
\sum_{n=0}^{\infty} dp(3n+1)q^{n} & = \frac{f_{2}f_{6}^3}{f_1^2f_{4}^3}. \label{eq3n+1}
\end{align}

Now we use \eqref{eq2} to extract the odd parts of \eqref{eq3n+1}, which gives us
\begin{align}
\sum_{n=0}^{\infty} dp(6n+4)q^{n} & = 2\frac{f_3^3f_8^2}{f_1^4f_2f_4} \nonumber \\
& \equiv 2f_2^3 f_3^3\pmod{4}. \label{eq28.5.1}
\end{align}

Using \eqref{Jacobi}, it follows from \eqref{eq28.5.1} that
\begin{align*}
\sum_{n=0}^{\infty} dp(6n+4)q^{n} & \equiv 2f_2^3f_3^3 \pmod{4}\\
& = 2 \sum_{k,l=1}^{\infty} (-1)^{k+l}(2k+1)(2l+1)q^{k(k+1)+3l(l+1)/2} \\
& \equiv 2\sum_{k,l=1}^{\infty} q^{k(k+1)+3l(l+1)/2} \pmod{4}.
\end{align*}
Thus, 
\begin{align*}
\sum_{n=0}^{\infty} dp(6n+4)q^{24n+15} &  \equiv 2\sum_{k,l=1}^{\infty} q^{6(2k+1)^2+(6l+3)^2} \pmod{4}.
\end{align*}
It follows that $dp(6n+4) \equiv 0 \pmod{4}$ if $24n+15$ is not of the form $6x^2+y^2$. 

Since $p \equiv 17 \mbox{ or } 23 \pmod{24}$, it follows that $\displaystyle \left( \frac{-6}{p}\right) = -1$, which implies that $\nu_p(N)$ is even if $N$ is of the form $6x^2+y^2$. Taking $n = p^{2k+1}m + 5({p^{2k+2}-1})/{8}$, we have 
$$24n+15 = 24p^{2k+1}m + 15p^{2k+2} = p^{2k+1}(24m+15p).$$
Therefore, $\nu_p(24n+15)$ is odd and the result follows by the fact that
$$6\left( p^{2k+1}m + \frac{5(p^{2k+2}-1)}{8} \right) + 4 = 6p^{2k+1}m + \frac{15p^{2k+2}+1}{4}.$$

\section{Proofs of Theorems \ref{Th3}--\ref{Th4}}
We now transition to proving Theorems \ref{Th3}--\ref{Th4} where the focus in on arithmetic properties modulo 3.  Let us begin by proving \eqref{c2}. It follows from \eqref{eq3n+1} that
\begin{align*}
\sum_{n=0}^{\infty} dp(3n+1)q^{n} & \equiv \frac{f_{2}f_{6}^3}{f_1^2f_{12}} \pmod{3}. 
\end{align*}
Thanks to \eqref{eq6}, we get
\begin{align*}
\sum_{n=0}^{\infty} dp(9n+7)q^{3n+2} \equiv 4q^2\frac{f_{6}^5f_{18}^3}{f_3^6f_{12}} \pmod{3}.
\end{align*}
Dividing by $q$ and replacing $q^3$ by $q$ yields 
\begin{align*}
\sum_{n=0}^{\infty} dp(9n+7)q^{n} \equiv \frac{f_{2}^5f_{6}^3}{f_1^6f_{4}} \equiv \frac{f_6^4f_2^2}{f_3^2f_4} \pmod{3}.
\end{align*}
Finally, we use \eqref{eq7} to get
\begin{align}
\sum_{n=0}^{\infty} dp(9n+7)q^{n} \equiv \frac{f_6^4}{f_3^2} \left( \frac{f_{18}^2}{f_{36}} - 2q^2\displaystyle\frac{f_{6}f_{36}^2}{f_{12}f_{18}} \right) \pmod{3},
\label{eq29.5.1}
\end{align}
from which \eqref{c2} follows.

We next consider the ``internal congruence'' given in Theorem \ref{internal_cong}.  It follows from \eqref{eq29.5.1} that
\begin{align*}
\sum_{n=0}^{\infty} dp(27n+7)q^{3n} & \equiv \frac{f_6^4f_{18}^2}{f_3^2f_{36}}  \pmod{3}.
\end{align*}
Replacing $q^3$ by $q$, we are left with
\begin{align*}
\sum_{n=0}^{\infty} dp(27n+7)q^{n} & \equiv \frac{f_2^4f_{6}^2}{f_1^2f_{12}} \\
& \equiv \frac{f_2f_{6}^3}{f_1^2f_{4}^3} \pmod{3}.
\end{align*}
Thus, it follows by \eqref{eq3n+1} that
\begin{align*}
\sum_{n=0}^{\infty} dp(27n+7)q^{n}  \equiv \sum_{n=0}^{\infty} dp(3n+1)q^{n} \pmod{3},
\end{align*}
which completes the proof.

Lastly, we use the previous two theorems to prove Theorem \ref{Th4}.
We prove this theorem by induction on $\alpha$.  Note that the case $\alpha=1$ has already been established in Theorem \ref{Th3} above.  We now use it as our base case in a proof by induction.   
Thus, we assume that, for some $\alpha\geq 1$ and all $n\geq 0$, 
$$
dp \left( 3^{2\alpha + 1}n + \frac{7 \cdot 9^\alpha + 1}{4} \right) \equiv 0 \pmod{3}.
$$
We then want to prove that 
$$
dp \left( 3^{2\alpha + 3}n + \frac{7 \cdot 9^{\alpha+1} + 1}{4} \right) \equiv 0 \pmod{3}.
$$

We note that
\begin{align*}
dp \left( 3 \left( 3^{2\alpha}n + \frac{3\cdot 7 \cdot 9^{\alpha-1} -1}{4} \right) +1 \right) & = dp \left(  3^{2\alpha+1}n + \frac{7 \cdot 9^{\alpha} -3}{4} +1 \right) \\
& = dp \left(  3^{2\alpha+1}n + \frac{7 \cdot 9^{\alpha} +1}{4} \right) \\
& \equiv 0 \pmod{3},
\end{align*}
by the induction hypothesis. Finally, by Theorem \ref{internal_cong}, we have
\begin{align*}
dp \left( 3^{2\alpha + 3}n + \frac{7 \cdot 9^{\alpha+1} + 1}{4} \right) & = dp \left( 3^{2\alpha + 3}n + \frac{7 \cdot 9^{\alpha+1} -27}{4} + 7 \right) \\
& = dp \left( 27 \left( 3^{2\alpha}n + \frac{3 \cdot 7 \cdot 9^{\alpha-1} - 1}{4} \right) + 7 \right) \\
& \equiv dp \left( 3 \left( 3^{2\alpha}n + \frac{3\cdot 7 \cdot 9^{\alpha-1} -1}{4} \right) +1 \right) \pmod{3}\\
& \equiv 0 \pmod{3},
\end{align*}
which concludes the inductive proof.

\end{document}